\ifdefined\devmode
\documentclass[a5paper,oneside,11pt]{article}
\usepackage[hmargin=0.3cm,vmargin=0.4cm]{geometry}

\newcommand{\Comments}{1}
\else
\documentclass[11pt]{article}
\usepackage[margin=1.3in]{geometry}
\usepackage[font=small,labelfont=bf]{caption}
\newcommand{\Comments}{0}
\fi

\usepackage[usenames,dvipsnames]{color}
\IfFileExists{pdfcolmk.sty}{\usepackage{pdfcolmk}}{} 

\usepackage[pdftex,colorlinks=true,linkcolor=blue,urlcolor=blue]{hyperref}

\usepackage{amsmath, amsfonts, amsthm, amssymb}
\usepackage{algorithmic, algorithm}
\usepackage{enumerate}
\usepackage{graphicx}
\usepackage{tikz}
\usepackage[center]{subfigure}
\usetikzlibrary{arrows}
\graphicspath{{figs/}}
\usepackage{kbordermatrix}

\newcommand{\tr}{\top}

\newcommand{\B}{\mathcal{B}}
\newcommand{\F}{\mathcal{F}}
\newcommand{\G}{\mathcal{G}}
\newcommand{\K}{\mathcal{K}}

\renewcommand{\P}{\mathcal{P}}

\newtheorem{theorem}{Theorem}[section]

\newtheorem{definition}[theorem]{Definition}

\newcommand{\ifcomments}[1]{\ifnum\Comments=1{#1}\fi}
\newcommand{\fnote}[1]{\ifcomments{\footnote{\textcolor{gray}{#1}}}} 

\title{Topological Entropy Bounds for Hyperbolic Plateaus of the\\ H\'{e}non Map}
\author{Rafael M. Frongillo\thanks{\url{raf@cs.berkeley.edu}.
Work done at Cornell University and the University of California at Berkeley.}
}

\begin{document}

\maketitle

\begin{abstract}
Combining two existing rigorous computational methods, for verifying hyperbolicity (due to Arai) and for computing topological entropy bounds (due to Day et al.), we prove lower bounds on topological entropy for 43 hyperbolic plateaus of the H\'{e}non map.  We also examine the 16 area-preserving plateaus studied by Arai and compare our results with related work.  Along the way, we augment the algorithms of Day et al. with routines to optimize the algorithmic parameters and simplify the resulting semi-conjugate subshift.
\end{abstract}


\section{Introduction} \label{sec:intro}

Dynamical systems theory has seen the emergence of many rigorous computational methods in recent years.  Such tools often extend the realm of provable theorems well beyond what is possible with chalk and blackboard\fnote{, or even manual use of popular computational mathematics environments such as Matlab or Mathematica}.  This is particularly true of the recent automated tools to compute topological entropy bounds~\cite{newhouse, dft} and to prove hyperbolicity~\cite{zinplats, hruska, mazur}.

Let us briefly recall the relevant characteristics of these methods.  Both techniques for proving entropy bounds first construct a subshift of finite type (SFT), whose topological entropy is easily computed and is a lower bound of that of the original system.  Newhouse et al. compute rigorous approximations of stable and unstable manifolds of periodic orbits, and then construct a SFT using pieces of these manifolds; in this regard, their technique could be considered a rigorous version of the trellis method.   Day, Frongillo, and Trevi\~{n}o (DFT)~\cite{dft} construct a discrete multivalued map from a discretization of the phase space, and then apply discrete Conley index theory to prove a semi-conjugacy to a particular SFT.  In certain settings, such as the one we study here, the DFT method is \emph{completely} automated, meaning that after a simple initialization, no further manual input is required.  It is unclear whether the method of Newhouse et al. shares this automation property.

Hruska~\cite{hruska} developed one of the first automated methods for rigorously verifying hyperbolicity, based on the computation of cone fields.  In contrast, Arai~\cite{zinplats} employs a more indirect technique using the notion of quasi-hyperbolicity.  His approach allows for more efficient computations than Hruska, but does not guarantee that the nonwandering set is not just a finite collection of periodic orbits, or even that it is nonempty.

In~\cite{zinplats}, Arai identifies several hyperbolic regions of the H\'enon map:
\begin{equation}
  f_{a,b}(x,y) = (a-x^2+by, x)
  \label{eq:henon-intro}
\end{equation}
These regions are dubbed \emph{hyperbolic plateaus} because topological entropy is constant across any such region.  Arai's technique does not reveal anything about these topological entropy values, however, and it is therefore natural to combine his computations with an automated method for proving topological entropy bounds.

In this paper we use the DFT method~\cite{dft} to compute lower bounds for topological entropy of the hyperbolic plateaus of H\'enon computed by Arai in~\cite{zinplats}.  The constant entropy on each plateau enables us to extend a lower bound computation from a single setting of parameter values to an entire region of the parameter space.  Additionally, the full automation of the DFT method enables us to study a total of 58 parameter values with essentially the same manual effort as studying one.  We use the DFT method as a black box, but we give a high-level overview of the approach in Section~\ref{sec:techniques}, including new techniques for improving robustness and simplifying the resulting SFT.

Theorems~\ref{thm:mainresult} and~\ref{thm:ap-bounds} summarize our results.  To the author's knowledge, all of these rigorous lower bounds are the largest known for their corresponding parameters.  We selected the parameter regions so that the bounds obtained might give a global picture of the entropy of the H\'{e}non map as a function of the parameters; see Figure~\ref{fig:plat-ents-3d} for such a picture.

We also study several of the area-preserving H\'{e}non maps in Section~\ref{sec:ap}, which have been well-studied in the Physics community~\cite{prunehenon, homoclinicbif, howtoprune}.  Recently, some precise rigorous results emerged as well~\cite{zinloops}.  We find that our rigorous lower bounds match or are very close to estimates given in previous work, and match the rigorous results exactly when applicable.

\begin{figure}[!ht] \centering
  \includegraphics[width = 0.8\textwidth]{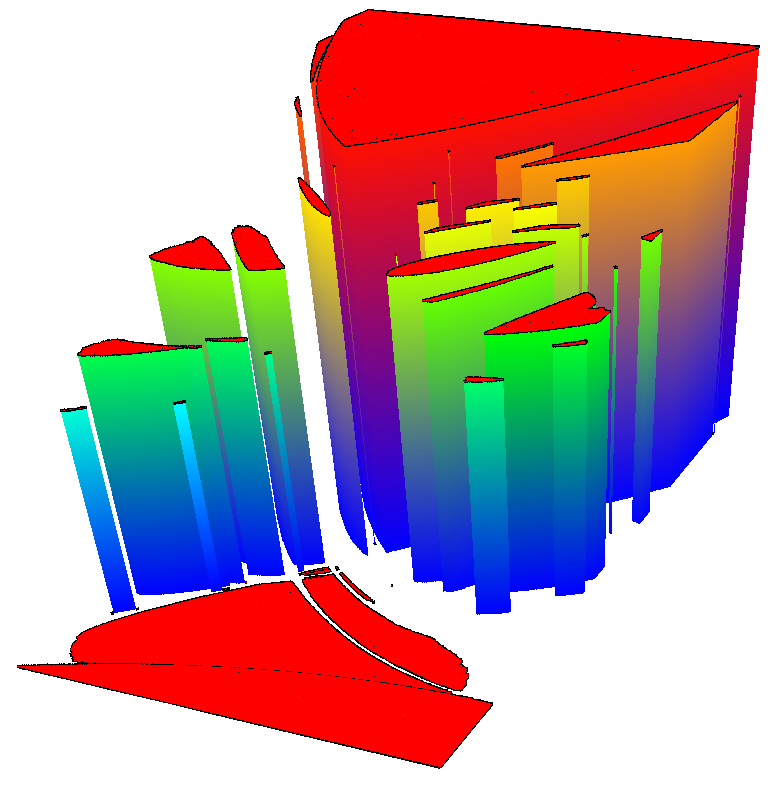}
  \vspace{-10pt}
  \caption{Rigorous lower bounds for topological entropy for the hyperbolic plateaus in Figure~\ref{fig:plat-nums}. The height of each plateau in the visualization is proportional to the entropy bound computed. See Theorem~\ref{thm:mainresult} or Figure~\ref{fig:plat-ents-flat} for the actual bounds.}
  \label{fig:plat-ents-3d}
\end{figure}

\section{Background} \label{sec:bg}

We first review basic definitions related to symbolic dynamics, topological entropy, and hyperbolicity, in Sections~\ref{sec:symbolics} and~\ref{sec:hyp}.  The DFT approach relies heavily on a combinatorial version of the discrete Conley index; in Section~\ref{sec:combinatorial}
we discuss the index at a high level, and introduce the combinatorial structures that relate it to our computational setting.

\subsection{Symbolic dynamics and topological entropy} \label{sec:symbolics}

Define a \emph{symbol space} $X_n = \{0,\dots,n-1\}^\mathbb{Z}$ to be the set of all bi-infinite sequences on $n$ symbols.  It is well-known that $X_n$ is a complete metric space.  Let the \emph{full n-shift} $\sigma: X_n\rightarrow X_n$ be the map acting on $ X_n$ by $(\sigma(x))_i = x_{i+1}$.  Given a directed graph $G$ on $n$ nodes with $n\times n$ transition matrix $A$ with $A_{i,j}\in\{0,1\}$, we can define an induced symbol space $X_G\subseteq X_n$, where $x\in X_G$ if and only if for $x = (\dots,x_i,x_{i+1},\dots)$, $A_{x_i,x_{i+1}} = 1$ for all $i$.  That is, $X_G$ consists of all sequences in $X_n$ with transitions of $\sigma$ allowed by the edges of $G$.  Equipped with the corresponding shift map $\sigma_G : X_G \to X_G$, we call $(X_G,\sigma_G)$ a \emph{subshift of finite type}.

We use topological entropy to measure the relative complexity of different dynamical systems.  If the topological entropy of a dynamical system $f$, denoted $h(f)$, is positive, we say that $f$ is {\em chaotic}.

\begin{definition}[\cite{coding}, adapted from~\cite{bowen}]
  Let $f:X\to X$ be a continuous map with respect to a metric $d$.  We say that a set $W\subseteq X$ is $(n,\varepsilon)$-separated under $f$ if for any distinct $x,y\in W$ we have $d\bigl(f^j(x),\;f^j(y)\bigr) >\varepsilon$ for some $0\leq j < n$.  The {\em topological entropy} of $f$ is
  \begin{equation}
	h(f) := \lim_{\varepsilon\to\infty} \; \limsup_{n\to\infty}
	\frac{log(s_f(n,\varepsilon))}{n},
  \end{equation}
  where $s_f(n,\varepsilon)$ denotes the maximum cardinality of an $(n,\varepsilon)$-separated set under $f$.
\end{definition}

While topological entropy can be difficult to calculate in general, there is a simple formula for subshifts of finite type which is given in the following theorem.  For a proof, see~\cite{coding} or~\cite{robinson}.

\begin{theorem}
  Let $G$ be a directed graph with transition matrix $A$, and let $(X_G,\sigma_G)$ be the corresponding subshift of finite type.  Then the topological entropy of $\sigma_G$ is $h(\sigma_G) = \log(\text{sp}(A))$, where $\text{sp}(A)$ denotes the spectral radius (maximum magnitude of an eigenvalue) of $A$.
\end{theorem}

When studying a complex map $f$, it is sometimes useful to study a subsystem of $f$ which can be precisely related to $f$ via a semi-conjugacy.

\begin{definition}
Let $f:X\rightarrow X$ and $g:Y\rightarrow Y$ be continuous maps for topological spaces $X, Y$. A \emph{semi-conjugacy} from $f$ to $g$ is a continuous surjection $\phi:X\rightarrow Y$ with $\phi\circ f = g\circ \phi$. We say that $f$ is \emph{semi-conjugate} to $g$ if there exists a semi-conjugacy from $f$ to $g$.  If additionally $\phi$ is a homeomorphism, then $f$ and $g$ are \emph{conjugate}.
\end{definition}

Particularly relevant to our setting is the following result.

\begin{theorem}[\cite{robinson}]
  \label{thm:topentsemiconj}
  Let $f$ and $g$ be continuous maps, and let $\phi$ be a
  semi-conjugacy from $f$ to $g$.  Then $h(f) \geq h(g)$.
\end{theorem}

Note that if $f$ and $g$ are conjugate, Theorem~\ref{thm:topentsemiconj} gives us $h(f) = h(g)$.  In other words, topological entropy is invariant under conjugacy.

\subsection{Hyperbolicity} \label{sec:hyp}

To begin we define uniform hyperbolicity.  Throughout the paper, we will refer to this property simply as hyperbolicity.  Let $X$ be a manifold.
\begin{definition}[\cite{guck}] A map $f:X\to X$ is said to be (uniformly) \emph{hyperbolic} if for every $x\in X$ the tangent space $T_x X$ for $f$ is a direct sum of stable and unstable subspaces; more precisely, if we have $T_x X = E^s(x) \oplus E^u(x)$, where $E^s(x)$ and $E^u(x)$ satisfy the following inequalities for some $C>0$ and $0<\lambda<1$, and for all $n\in\mathbb{N}$:
  \begin{enumerate}\setlength{\itemsep}{0pt}
  \item $\|Df^n(v)\| \leq C\lambda^n\|v\|$ for all $v\in E^s(x)$.
  \item $\|Df^{-n}(v)\| \leq C\lambda^n\|v\|$ for all $v\in E^u(x)$.
  \end{enumerate}
\end{definition} This structure can be thought of as a generalization of the structure of the Smale horseshoe, namely that there are invariant directions, and there is uniform contraction and expansion in the stable and unstable directions, respectively.  Some useful properties of hyperbolic systems are discussed below, but for more details see~\cite{robinson} and~\cite{guck}.

An important property of hyperbolic maps is that they are structurally stable~\cite{robinson}, which implies that all maps in the same hyperbolic region are conjugate.  Thus, by Theorem~\ref{thm:topentsemiconj}, the topological entropy is constant within such a region.  For this reason, we will henceforth call these regions \emph{hyperbolic plateaus}.

Hyperbolicity often makes it easier to identify interesting dynamics, but it is important to note that sometimes a system can be ``vacuously'' hyperbolic, in the sense that it is hyperbolic but there is no recurrent behavior.  A helpful concept in this context is the \emph{nonwandering set}.

\begin{definition}[\cite{robinson}] The \emph{nonwandering set} of a map $f$ is the set of points $x$ for which every neighborhood $U$ of $x$ has $f^n(U)\cap U \neq \varnothing$ for some $n\geq 1$.
\end{definition}

\subsection{Combinatorial structures and the discrete Conley index}
\label{sec:combinatorial}

As we will be performing rigorous computations, we will naturally be interested in finite representations of the continuous map $f:X\to X$ in question.  Our main tool in this regard is a \emph{combinatorial enclosure}, a discrete ``outer-approximation'' of $f$.  We first discuss how we discretize the phase space $X$, which we assume here to be a subset of $\mathbb{R}^n$, though the constructions that follow generalize to more complicated spaces as well.

We begin by setting up a grid $\G$ on $X$, composed of finitely many cubical complices $\B_i$.  In practice, all elements of the grid are rectangles, which we call \emph{boxes}, represented as products of intervals (viewed in some nice coordinate chart); that is, for each $\B_i\in\G$, we can write $\B_i = \prod_{k=1}^n[x^i_k,y^i_k]$ for some $x^i,y^i\in\mathbb{R}^n$.  In the present paper, for each $k$ the interval widths $y^i_k - x^i_k$ are constant for all $i$, meaning the shape of all boxes is the same.  For a collection of boxes $\K\subseteq\G$, we denote by $|\K|$ its \emph{topological realization}, that is, its corresponding subset of $X$.

In our setting, we create $\G$ by selecting one box $\B$, called the \emph{initial box}, which encloses the entire area we wish to study.  We then subdivide $\B$ evenly $d$ times in each coordinate direction, meaning we replace $|\B|$ with $2^n$ boxes by splitting each coordinate in half, and then we recurse on each of those boxes, with total recursion depth $d$.  This gives a grid $\G$ with $2^{dn}$ boxes, each of size $2^{-d}$ relative to the initial box $\B$.  The integer $d$ will be refered to as the \emph{resolution} or \emph{depth}.  

After discretizing the phase space, we must discretize the map $f$ itself.  To do this we introduce the combinatorial enclosure.
\begin{definition}
A \emph{multivalued map} $F: X\rightrightarrows X$ is a map from $X$ to its power set, so that $F(x)\subseteq X$. If $F$ is acyclic\footnote{$F$ is acyclic if for all $x$, $F(x)$ has trivial homology on all levels (no ``holes'').}  and we have $f(x)\in F(x)\; \forall x\in X$ for some continuous single-valued map $f$, then $F$ is an \emph{enclosure} of $f$.
\end{definition}
\begin{definition}
\label{def:comb-enc}
Given an enclosure $F$ of $f$, a \emph{combinatorial enclosure} of $f$ is a multivalued map $\F:\G\rightrightarrows \G$ defined by
$\F(\B) = \{\B'\in\G:|\B'|\cap F(|\B|)\neq\varnothing\}$.
\end{definition}
The reason multivalued maps and enclosures are used in our computations is that, if constructed properly, they enable rigorous results.  In particular, for each $\B\in\G$ we use the interval arithmetic library Intlab~\cite{intlab} to compute rigorous outer approximations $Y_\B$ of $f(|\B|)$, and let $F(x) = Y_\B$ for all $x\in|\B|$.  Defining $\F$ as in Definition~\ref{def:comb-enc}, this technique keeps track of the error terms in the computations of the image of a box, ensuring that the true image $f(|\B|)$ is contained in $|\F(\B)|$.  Note that $|\F|$ is also an enclosure of $f$.

In this combinatorial setting, we can apply a version of the discrete Conley index to gather rigorous information about the dynamics of the underlying map.  We will not define the Conley index formally here, giving instead a high-level overview of the concepts immediately relevant to our study.  For a full treatment of Conley index theory, see~\cite{salamon,m-mrozek}; for the definitions and algorithms needed in our discrete setting, see~\cite{dft}.

Given some compact set $N\subseteq X$, its \emph{maximal invariant set} is the set of points $x\in N$ such that there exists a trajectory of $f$ through $x$ which stays entirely within $N$ for all forward and backward time.\footnote{That is, there exists some 
$\{x_n\}_{n\in\mathbb{Z}} \subseteq N$
such that $f(x_n) = x_{n+1}$ for all $n\in\mathbb{Z}$ and $x_0 = x$.}  A central concept of Conley index theory is that of \emph{isolation}: we say that an invariant set $S$ (that is, a set with $f(S)=S$) is \emph{isolated} if there is a compact set $N$, called the \emph{isolating neighborhood} of $S$, such that $S$ is contained in the interior of $N$ and $S$ is the maximal invariant set of $N$.  From this notion, a \emph{(combinatorial) index pair} for the enclosure $\F$ is a pair $(\P_1,\P_0)$ of sets of boxes satisfying a two properties: (a) letting $P_i = |\P_i|$, the set $\overline{P_1\setminus P_0}$ must be an isolating neighborhood, and (b) if we contract the set $P_0$ to a point denoted $[P_0]$, the map $f$ induces on the pointed quotient space $(P_1/P_0,[P_0])$ must be continuous.\footnote{We define the induced map $g$ on this space as follows: define $g([P_0]) = [P_0]$, and for $x\in P_1\setminus P_0$, define $g(x)=[P_0]$ if $f(x)\in P_0$ and $g(x) = f(x)$ otherwise.}
Finally, the \emph{Conley index} of the index pair $(\P_1,\P_0)$ is a particular topological invariant of its maximal invariant set based on relative homology; in our setting, the index can be represented as a finite sequence of matrices describing the map that $f$ induces on each level of homology.\footnote{Specifically, the index is the shift equivalence class of the induced map on relative homology groups of the quotient space.    Again, for the full details, see~\cite{salamon,m-mrozek,dft}.}

With the Conley index, we can make strong statements about the behavior of $f$.  For example, if the Conley index corresponding to index pair $(\P_1,\P_0)$ is nontrivial,\footnote{The index is trivial if it is the shift equivalence class of the zero map.} the maximal invariant set of $\overline{P_1\setminus P_0}$ is nonempty.
Along the same lines, but using a more nuanced analysis, the DFT method is able to take the Conley index and compute a SFT which is semi-conjugate to $f$.
We describe the DFT method at a high level in Section~\ref{sec:dft}.  Note however that computing the Conley index in the form of an induced map on homology is a difficult task in and of itself.  For a thorough treatment of computational homology, dealing with this and other applications, see~\cite{CompHom}.  For our computations, we use the \emph{homcubes} package~\cite{homcubes}, part of the software package \emph{CHomP}~\cite{chomp}.

\section{Simplifying Subshifts} \label{sec:shifteq}
Given a subshift of finite type $(X_G,\sigma_G)$ for a graph $G$, it is often of interest to know whether there is a graph $H$ on fewer vertices such that $(X_G,\sigma_G)$ and $(X_H,\sigma_H)$ are conjugate.  To this end, we recall the notion of strong shift equivalence.

\begin{definition}[Strong shift equivalence]
  Let $A$ and $B$ be matrices.  An {\em elementary shift equivalence} between $A$ and $B$ is a pair $(R,S)$ such that
  \begin{equation}
	A = RS \text{   and   } B = SR.
  \end{equation}
  In this case, we write $(R,S):A\to B$.  If there is a sequence of such elementary shift equivalences $(R_i,S_i):A_{i-1}\to A_{i}, \; 1\leq i\leq k$, we say that $A_0$ and $A_k$ are \emph{strongly shift equivalent}.
\end{definition}

This notion is useful because of the following result due to R. F. Williams.

\begin{theorem}[\cite{williams1973classification}]
  \label{thm:shifteqconj}
  For directed graphs $G$ and $H$, the corresponding subshifts $(X_G,\sigma_G)$ and $(X_H,\sigma_H)$ are conjugate if and only if the transition matrices of $G$ and $H$ are strongly shift equivalent.
\end{theorem}

Theorem~\ref{thm:shifteqconj} allows us to prove that two subshifts are conjugate by a series of simple matrix computations.  Finding matrices that give a strong shift equivalence, however, can be a very difficult problem.  Two methods of finding such equivalences are given in~\cite{coding}: state splitting, where a single vertex is split into two, or state amalgamation, where two vertices are combined into one.  In graph-theoretic terms, amalgamating two vertices is equivalent to contracting them together.  In general obtaining the smallest element of a strong shift equivalence class may involve both splittings and amalgamations.  We instead focus on the simpler problem of obtaining $H$ only by amalgamating vertices in $G$.  This also has the advantage of producing a matrix that is more useful for our needs in this paper (see Section~\ref{sec:techniques}).

Let $A$ be the binary $n\times n$ transition matrix for $G$.  The following two conditions, adapted from~\cite{coding}, will allow us to amalgamate vertices $i$ and $j$.
\begin{eqnarray}
  \text{\bf Forward Condition:} & 
  A \vec e_i = A \vec e_j \text{ and } (\vec e_i^\tr  A) \cdot (\vec e_j^\tr A) = 0
  & \label{eq:fwd} \\
  \text{\bf Backward Condition:} &
  \vec e_i^\tr A = \vec e_j^\tr A \text{ and } (A \vec e_i) \cdot (A \vec e_j) = 0 
  & \label{eq:backwd}
\end{eqnarray}
Here $\vec e_i$ denotes the column vector with a 1 in position $i$ and zeros elsewhere.  From a graph-theoretic or dynamical systems point of view, the forward condition says that $i$ and $j$ have the same image but disjoint preimages, and the backward condition says they have the same preimage but disjoint images.  See Figure~\ref{fig:amalg-example} for an example.

\begin{figure}[ht]
  \centering

  \tikzstyle{fig}=[->,>=triangle 45,scale=0.4,node distance=1mm,semithick]
  \begin{tikzpicture}[fig]
 \tikzstyle{gen}=[fill=black!10,draw=black,shape=circle,minimum size=2pt]
    \tikzstyle{gen2}=[fill=blue!10,draw=blue,shape=circle,minimum size=2pt]

    \node[gen2] (a) at (2,0) {$a$};
    \node[gen] (b) at (0,3) {$b$};
    \node[gen2] (c) at (4,3) {$c$};
    
    \path
    (a) edge[loop below] (a)
        edge (b)
    (b) edge[bend left] (c)
    (c) edge[bend left] (b)
        edge (a)
    ;

    \node at (6,1.5) {\huge $\Rightarrow$};
  \end{tikzpicture}
  \hspace{10pt}
  \begin{tikzpicture}[fig]
    \tikzstyle{gen}=[fill=black!10,draw=black,shape=circle,minimum size=2pt]

    \node[gen] (b) at (0,3) {$b$};
    \node[gen] (ac) at (4,3) {$ac$};
    
    \path
    (ac) edge[loop below] (ac)
        edge[bend left] (b)
    (b) edge[bend left] (ac)
    ;

    \node at (0,0) {}
    ;
  \end{tikzpicture}
  \caption{A forward amalgamation}
  \label{fig:amalg-example}
\end{figure}
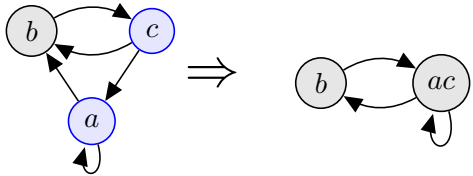

Note that the backward condition for $A$ is the same as the forward condition for $A^\tr$.  The following result allows us to reduce $A$ to a smaller $n-1$ by $n-1$ matrix $B$ if either of these conditions are satisfied for some pair of vertices.  See e.g.~\cite[\S 2]{coding} for a proof.

\begin{theorem}
  \label{thm:amalg}
  If $i$ and $j$ satisfy the forward condition \eqref{eq:fwd} or backward condition \eqref{eq:backwd} for a transition matrix $A$, then there is an elementary shift equivalence from $A$ to the matrix obtained by amalgamating $i$ and $j$.
\end{theorem}

By applying Theorem~\ref{thm:amalg} repeatedly, as long as there exist $i,j$ satisfying either the forward or backward condition, one can reduce $A$ to a much smaller representative of its strong shift equivalence class.  The resulting matrix $B$ at the end of this process corresponds to a subshift $(X_H,\sigma_H)$ which is therefore conjugate to $(X_G,\sigma_G)$.

\begin{algorithm}[!ht]
\caption{{\tt simplify\_subshift}: Amalgamating a subshift of finite type}
\begin{algorithmic}
  \STATE Input: subshift $T \in \{0,1\}^{n\times n}$, number of trials $K$
  \STATE $T_\text{min} \gets T$
  \FOR{$k$ from 1 to $K$}
  \STATE $\pi \gets {\tt random\_permutation}(n)$
  \STATE $T_\pi \gets T(\pi,\pi)$\quad \COMMENT{relabel vertices}
  \REPEAT
  \STATE {\tt amalgamated} $\gets$ {\tt false}
  \FOR{$(i,j) \in E(T_\pi)$ ordered lexographically}
  \IF{conditions~\eqref{eq:fwd} or~\eqref{eq:backwd} hold for $A=T_\pi$}
  \STATE $T_\pi\gets$ {\tt amalgamate}$(T_\pi,i,j)$
  \STATE {\tt amalgamated} $\gets$ {\tt true}
  \STATE {\bf break for}
  \ENDIF
  \ENDFOR
  \UNTIL{{\bf not} {\tt amalgamated}} \quad \COMMENT{no further amalgamations}
  \IF{${\tt size}(T_\pi) < {\tt size}(T_\text{min})$}
  \STATE $T_\text{min} \gets T_\pi$
  \ENDIF
  \ENDFOR
  \STATE Output: $T_\text{min}$
\end{algorithmic}
\label{alg:amalg}
\end{algorithm}

For small enough matrices it is feasible to perform a simple brute-force search to find the smallest $B$ which can be obtained from $A$ via amalgamations, but we would like a more efficient algorithm for larger matrices.  Unfortunately, it is shown in~\cite{amalg} that it is NP-hard (computationally intractable) to find an ordering of amalgamations which yields the smallest representative.  In light of this hardness, we use the procedure outlined in Algorithm~\ref{alg:amalg}, which is a essentially a randomized greedy algorithm performed many times.  Taking the number of trials $k$ to be about $n^2$ will typically give a reasonable approximation factor, meaning that if $m$ amalgamations are possible, the algorithm will perform roughly $m/3$ or $m/2$ amalgamations.  It remains an open question whether an algorithm exists which has a provable approximation guarantee.

\section{Techniques} \label{sec:techniques}

Given a continuous map $f:X\to X$, we will apply discrete Conley index theory to compute subshifts of finite type to which $f$ is semi-conjugate.  This is the approach of the DFT method~\cite{dft}, which we describe first in Section~\ref{sec:dft}.  We also add some new tools which we present in Sections~\ref{sec:reducing} and~\ref{sec:scaling}.  Finally, we discuss implementation and efficiency details in Section~\ref{sec:implementation}.

\subsection{The DFT method} \label{sec:dft}

To obtain lower bounds on topological entropy, we compute a semi-conjugate symbolic dynamical system using the DFT method, which is based on the discrete Conley index.  We use this method largely as a black box, and hence give only a rough sketch of the technique here and refer the reader to Day et al.~\cite{dft} for the full details.  At a high level, the method consists of three main steps:
\begin{enumerate}[I.]\setlength{\itemsep}{0pt}
\item Discretize a compact subset of the phase space by constructing a grid $\G$ of boxes, and compute a rigorous enclosure for the dynamics on these boxes as described in Section~\ref{sec:combinatorial}.
\item Find a combinatorial index pair from these boxes and determine the Conley index for this pair using~\cite{homcubes}.
\item From the index, compute a subshift to which the original system is semi-conjugate.\label{item:dft-subshift}
\end{enumerate}
This approach is very general, and in principle could be applied to systems of arbitrary dimension.  A major benefit to using this method here, however, is that in our setting it is completely automated.  As long as one knows roughly where in the phase space the invariant objects of interest are, one can simply plug in the parameters and compute.  See Section~\ref{sec:implementation} for more details on the implementation.

An advantage to studying hyperbolic parameters of H\'{e}non map~\eqref{eq:henon-intro} for $|b| \leq 1$ is that the nonwandering set is disconnected.  This follows from Plykin theory, as discussed in~\cite[\S 7.9]{robinson}, from which we know that any connected trapping region (a region $N$ with $f(N) \subseteq \text{int}(N)$) of the attractor has at least three holes.  Considering a disc covering such a hole, we see that since the image each hole must strictly cover another hole, an iterate of this disk must eventually expand, which contradicts the area-preserving or area-shrinking of the maps we are considering.  Thus, all such trapping regions must be disconnected.

Since the nonwandering set is disconnected, we can bypass much of the complication in the second step of the DFT method, that of finding an index pair.  This is because the invariant set will be naturally isolated; at a fine enough resolution, the collection of boxes that cover the invariant set will already be separated into disjoint regions.

\subsection{Reducing large subshifts} \label{sec:reducing}

When applying the DFT method to complex systems, the resulting semi-conjugate subshift (step~\ref{item:dft-subshift} above) is often very large.  In fact, the main result in~\cite{dft} is shown via a subshift with 247 symbols.  While such large subshifts provide a topological entropy bound and information about the number of orbits of a given period, they often carry little \emph{intuition} about the underlying structure of the dynamics.  The question then becomes, how can we distill more useful information from these large subshifts to get a more intuitive understanding of the system?

The answer we propose here is to simplify the resulting subshift using amalgamations, as described in Section~\ref{sec:shifteq}.  Specifically, we add a final step to the DFT method, where we run the semi-conjugate subshift obtained through Algorithm~\ref{alg:amalg}.  As we will see in Section~\ref{sec:results}, in practice this procedure can greatly reduce the number of symbols required to describe the system.  In some cases this simplification reveals a simple underlying structure, such as connections between a handful of low-period orbits, which would otherwise be difficult to glean from the 200+ symbols of the original subshift.  This simplification is also useful in comparing our results to previous work and conjectures; using Algorithm~\ref{alg:amalg}, one can attempt to amalgamate our computed subshift $A$ to yield a target subshift $B$.  We apply this technique in Section~\ref{sec:ap} to partially confirm a conjecture of Davis et al.~\cite{mackay}; see Figure~\ref{fig:dms-amalg} in particular.

As a final note, recall that the subshift generated by the DFT method has a useful geometric interpretation: from~\cite{dft} we know that we can associate a region $N_i$ of the phase space to each symbol $s_i$ of the subshift $A$, such that any trajectory $(\ldots, s_{-1}, s_0, s_1, \ldots)$ in $A$ corresponds to a trajectory in the original system through the regions $(\ldots, N_{-1}, \allowbreak N_0, \allowbreak N_1, \ldots)$.\footnote{In fact, these regions are connected components of the maximal invariant set of the index pair.}  Fortunately, amalgamation (and thus Algorithm~\ref{alg:amalg}) preserves this property in the following sense.  If $A'$ is derived from $A$ by a sequence of amalgamations, each symbol $s_i'$ of $A'$ can be expressed as an amalgamation of symbols of $A$.  Taking $N_i'$ to be the union of the regions corresponding to these symbols, it is easy to see that the same trajectory property holds for $A'$.

\subsection{Robustness and scaling} \label{sec:scaling}
\begin{figure}
  \centering
  \subfigure[center][Aspect ratio: entropy vs. $\log(w_2/w_1)$ \protect\newline Depths 9 (blue) and 10 (green)]{
    \includegraphics[width = 0.48\textwidth]{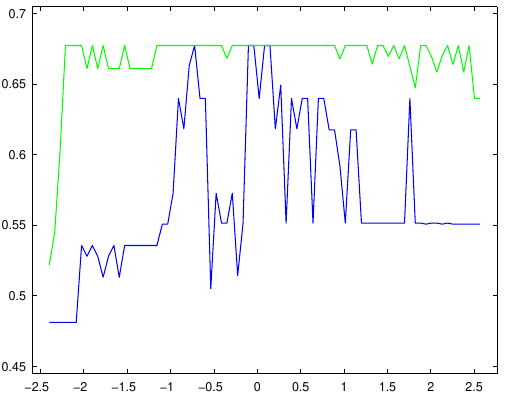}
    \label{fig:ratio-scale}
  }
  \subfigure[Resolution: entropy vs. $-\log(w_1w_2)$]{
    \includegraphics[width = 0.47\textwidth]{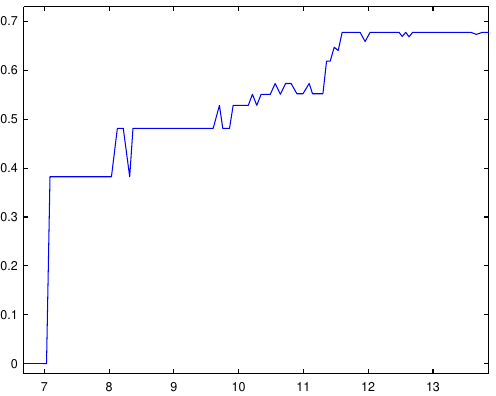}
    \label{fig:area-scale}
  }
  \vspace{-10pt}
  \caption{Sensitivity of the entropy bound with respect to the aspect ratio (a) and the area (b) of the boxes.  The map used was H\'enon~\eqref{eq:henon-intro} with parameters (5.685974, -1), or plateau 14 from Section~\ref{sec:ap}.}
  \label{fig:scaling}
\end{figure}

A natural concern for any approach which involves discretizing the phase space is the robustness of the method with respect to the choice of discretization.  In our setting, this discretization is determined by the resolution (depth), aspect ratio,\footnote{We will use this term to refer to the shape of the boxes in higher dimensions as well.} and the precise placement of the grid $\G$.

Ideally, slight changes in these scaling parameters would have at most minor effects on the resulting subshift or entropy bound, but there are several reasons why this is unreasonable to expect.  Perhaps most obvious is that a subshift of finite type, and even the Conley index itself, is a discrete object, and one cannot expect any sort of continuity in the scaling parameters when the output itself is discrete.  Moreover, the combinatorial isolation of the index pair depends on the topological properties of rigorous numerical bounds on the images of boxes, which can be very sensitive to the precise grid parameters.  A simple example is in achieving disjoint regions; when trying to separate regions $A$ and $B$ of the phase space with $\text{dist}(A,B) = d$, using a grid whose boxes $\B_i$ have width $y_k^i - x_k^i = 4d$ in each dimension $k$, great care must be taken in placing the grid, or the entire index pair could degenerate.  More subtle is a situation where the combinatorial image of $A$ does not overlap $B$, but does after a slight shift or rescaling of the grid.  These issues raise the question of the practical robustness of the DFT method with respect to these scaling parameters.

To measure this robustness, we plot the computed entropy lower bound for a particular H\'enon map against the area and the aspect ratio of the boxes in Figure~\ref{fig:scaling}.  There we define $w_1 = y^i_1 - x^i_1$ and $w_2 = y^i_2 - x^i_2$ to be the width of a box $\B_i$ in dimensions 1 and 2, respectively (recall that in our setting these are independent of $i$).  The grid resolution may be defined as $-\tfrac 1 2\log_2 w_1w_2$, or more generally as $-\tfrac 1 n\log_2 V(\B_i)$, where $V$ denotes the $n$-dimensional volume.  We choose the this formula so that after normalizing by the volume of the initial box $\B$, the notions of resolution and depth align.

Of course, the DFT method may behave very differently on other maps, but the behavior shown is quite typical: the entropy lower bound is roughly monotone increasing with respect to the resolution, and (very) roughly unimodal with respect to the aspect ratio.  This behavior is not surprising; the error from discretization decreases with the box area, and both extremes of the aspect ratio (i.e. $w_1 >\!> w_2$ and $w_2 >\!> w_1$) should result in essentially 1-dimensional information and hence a trivial entropy lower bound for most maps.

While the DFT method appears to be relatively robust with respect to the grid resolution, Figure~\ref{fig:ratio-scale} clearly shows a high sensitivity to the aspect ratio.  Specifically, small changes in the aspect ratio at depth 9 resulted in large jumps in the entropy lower bound, even when close to the optimal ratio.  While the behavior at depth 10 is somewhat more typical, this sensitivity is still something to keep in mind.  In particular, for the sake of replication, care should be taken when altering and storing the aspect ratio.\footnote{To find an appropriate aspect ratio for a given map, one can perform a scaling parameter exploration at a lower depth, similar to the one in Figure~\ref{fig:ratio-scale}.}\textsuperscript{,}\footnote{Note also that one can scale a map in a non-constant fashion to better align it with the grid.  Consider the map $f(x,y) = \bigl( (a - (\alpha^{y} x)^2 + b y) \alpha^{-\alpha^{y} x},\; \alpha^{y} x \bigr)$, which is just the H\'enon map~\eqref{eq:henon-intro} conjugated by $g(x,y) = (\alpha^{y} x, y)$, and thus for large enough $\alpha$ would be more aptly analyzed by first conjugating by $g^{-1}$.  More generally, such non-constant scaling may be used to focus on areas af the phase space where isolation is more difficult.}

As mentioned above, the DFT method gives an entropy bound which is roughly monotone increasing with the resolution.  Similar plots shown in~\cite{dft} suggested monotonicity, but had very few data points; it is likely that a continuous resolution scaling would fill in these plots to reveal a rough monotonicity in that case as well.
This monotonicity is of course beneficial behavior, as one would like the precision of the bound to increase with the precision of the box covering, but it is especially useful in hyperbolic settings.  By~\cite[Theorem~9.6.1]{robinson}, hyperbolic systems admit a finite Markov partition of the invariant set, and since the nonwandering set is disconnected (see Section~\ref{sec:dft}), in theory this partition is obtainable when the boxes are small enough.  Thus we expect to obtain the true entropy value at a high enough resolution, and by seeing where the entropy levels off we can be confident, though not certain, that our lower bound is the actual value.  We will apply this intuition in Section~\ref{sec:results}.

\subsection{Implementation and efficiency} \label{sec:implementation}

The computations in this paper were performed in Matlab, using Intlab~\cite{intlab}, GAIO~\cite{gaio}, and CHomP~\cite{chomp}.  The machines used had memory between 1 and 2 gigabytes and clock speeds between 1 and 2.2 gigahertz.  The runtimes varied, as we will describe below, but ranged from a few seconds at low grid resolutions to several hours for high resolutions.

As mentioned above, it is useful to observe how the entropy bound changes as we increase the grid resolution, but of course this procedure is not without cost.  One would expect the runtime of the bound computation to increase, perhaps dramatically, with the resolution of the discretization.  Empirically, the running time seems to grow very roughly as $n^r$ where $n$ is the dimension of the invariant set, and $r$ is the resolution of the grid, as defined above.

Fortunately, as we have mentioned several times, the DFT method and the new techniques we add here are all completely automated.  By this we mean that one needs only to specify the map and the parameters to be studied (and the discretization parameters discussed above), and the rest of the computation, all the way to the semi-conjugate subshift, is done without any further human action or input.  Thus, although the computations may be time-consuming for high resolutions, it is computation time, not human time.

This lack of manual intervention enables vast explorations of map and discretization parameters.  For this paper, a total of 58 parameter values of the H\'enon map were studied: 43 in Section~\ref{sec:goodplats} and 15 in Section~\ref{sec:ap}.  For each of the 58 parameter values, we computed entropy bounds at 64 or 72 different resolutions, yielding a total of over 3900 separate computations.  This exploration would have been infeasible without complete automation (or a large team of researchers).

\begin{figure}
  \centering
  \vspace{-10pt}
  \includegraphics[width = 0.8\textwidth]{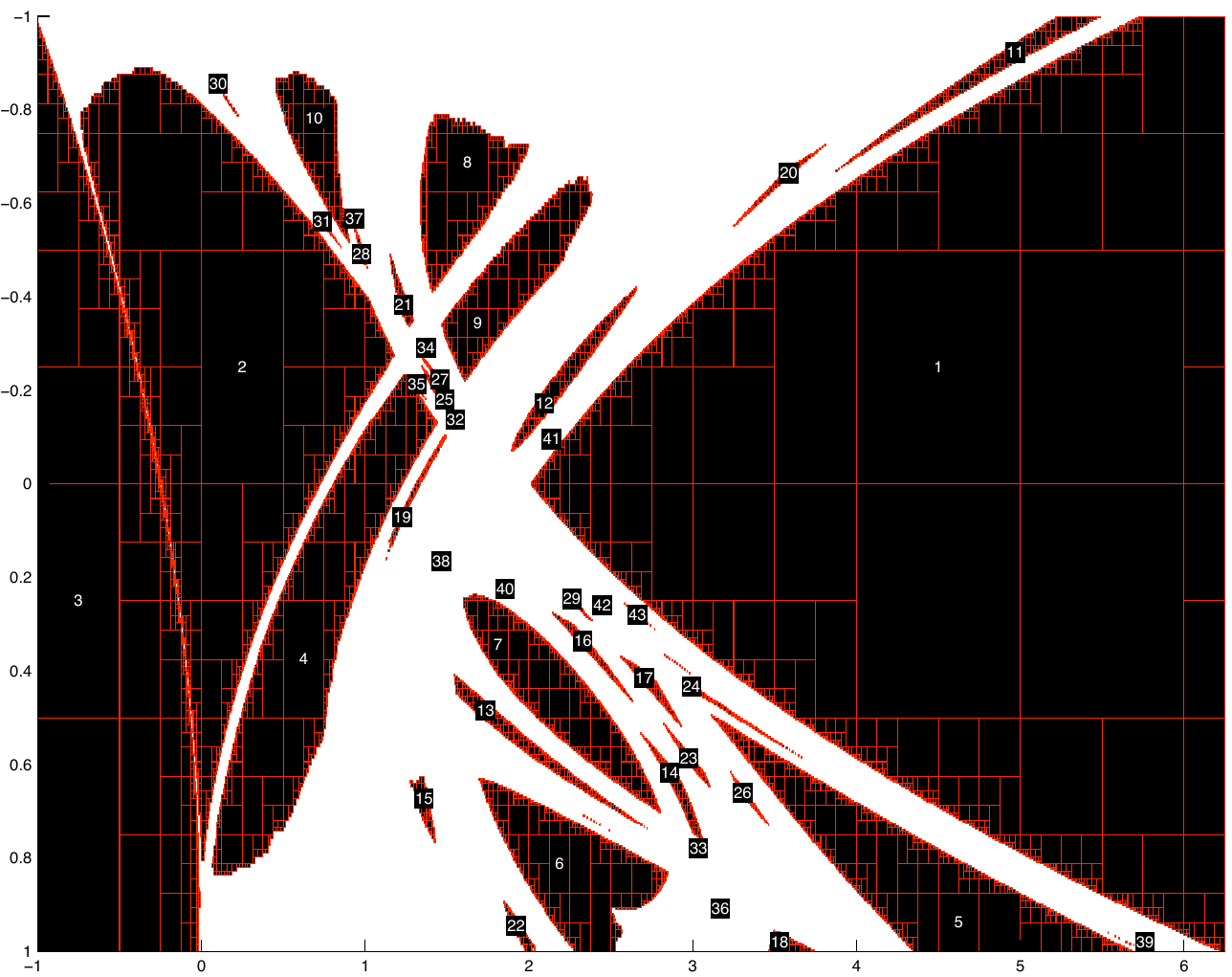}
  \vspace{-10pt}
  \caption{Hyperbolic plateaus for H\'{e}non from~\cite{zinplats}, with $a$ on the horizontal axis and $b$ on the vertical axis.  The label for each plateau is centered over the parameter values used to represent the plateau.}
  \label{fig:plat-nums}
\end{figure}

\begin{figure}
  \centering
  \vspace{-10pt}
  \includegraphics[width = 0.8\textwidth]{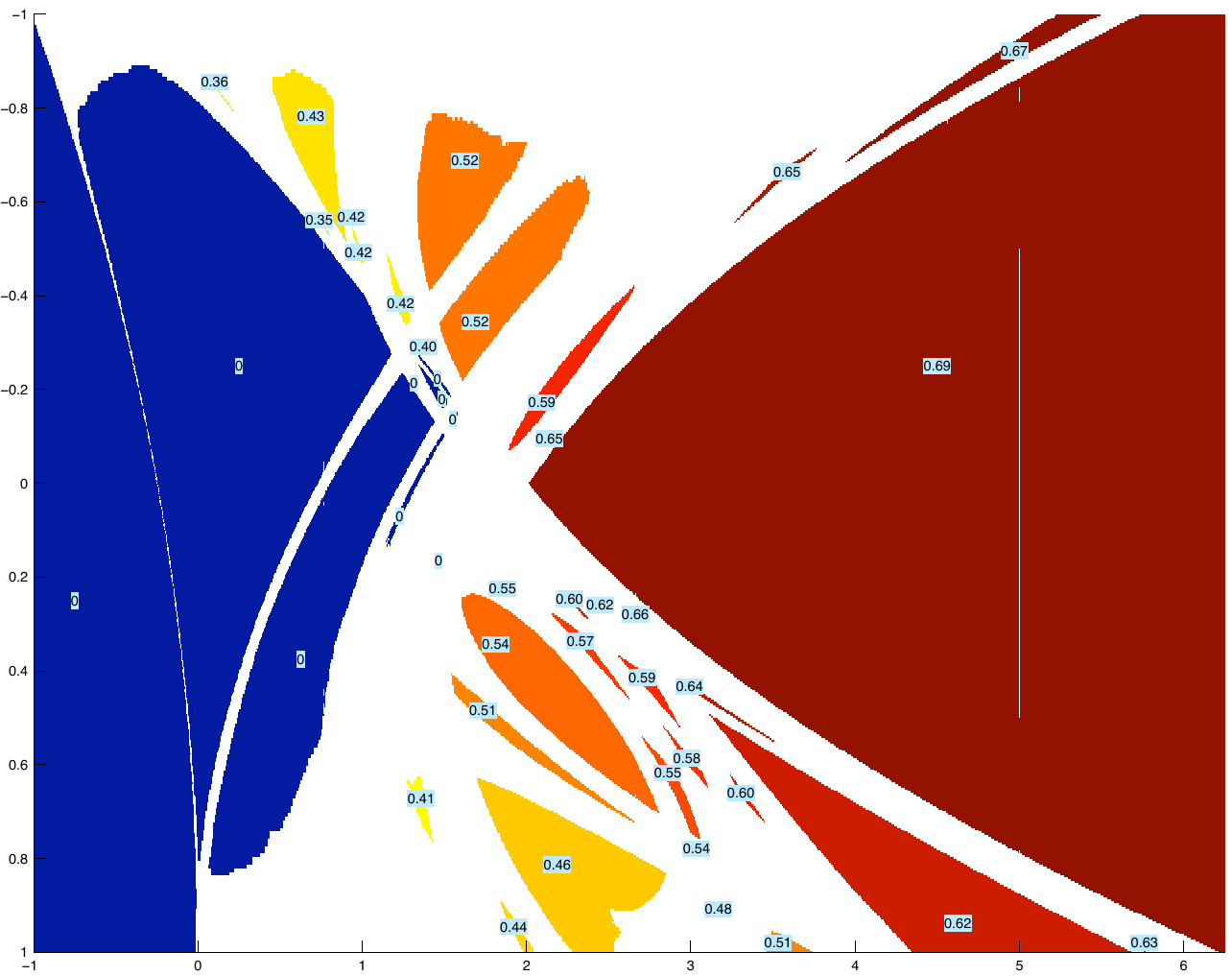}
  \vspace{-10pt}
  \caption{Rigorous lower bounds for topological entropy for the hyperbolic plateaus labeled 1 through 43 in Figure~\ref{fig:plat-nums}.}
  \label{fig:plat-ents-flat}
\end{figure}

\begin{figure}
  \centering
  \includegraphics[width = 0.75\textwidth]{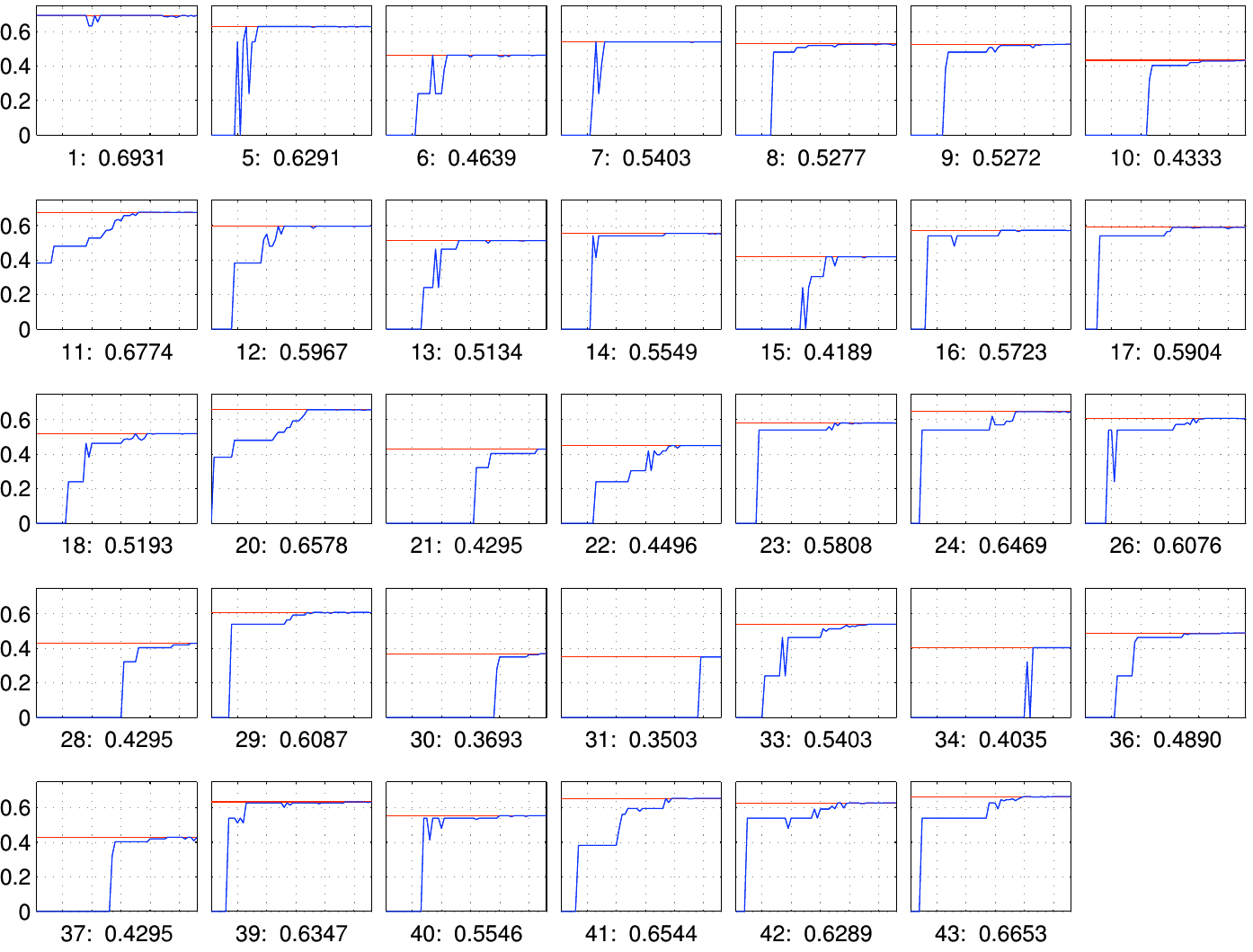}
  \caption{Entropy versus resolution.  The red horizontal line in each represents the maximum lower bound computed (also printed below each plot).  Plots are omitted for 0-entropy plateaus (2, 3, 4, 19, 25, 27, 32, 35, 38).}
  \label{fig:ent-vs-res-2d}
\end{figure}

\section{Hyperbolic plateaus of H\'{e}non} \label{sec:goodplats} \label{sec:results}
We now apply the methods outlined in Section~\ref{sec:techniques} to the real-valued H\'{e}non map
\begin{equation*}
  f_{a,b}(x,y) = (a-x^2+by, x)
  \label{eq:henon}
\end{equation*}
for parameter values $(a,b)$ such that $f_{a,b}$ is (uniformly) hyperbolic.  Note that this \emph{excludes} the classical parameters $(1.4, 0.3)$; for rigorous topological lower bounds in the classical case, see~\cite{dft} and~\cite{newhouse}.

Using the hyperbolic plateaus of Arai, we select representative parameter values to study for each plateau, as shown in Figure~\ref{fig:plat-nums}.  For each plateau, we use the continuous resolution scaling approach mentioned in Section~\ref{sec:scaling}, to get a sense of how close our bounds might be to the actual values.  The entropy bounds we compute constitute our main result, summarized in Theorem~\ref{thm:mainresult}.

\begin{theorem}
  \label{thm:mainresult}
  \vbox{Let $F_i = \{f_{a,b}\:|\:(a,b)\in R_i\}$, where $R_i$ is the $i$th plateau in Figure~\ref{fig:plat-nums}.  Then for all $i$ and all $f\in F_i$ we have $h(f) \geq h_i$, where the $h_i$ are defined below:}\nobreak
  \begin{displaymath}
    \begin{matrix}
      h_{1}  = 0.6931 & h_{5}  = 0.6291 & h_{6}  = 0.4639 & h_{7}  = 0.5403 & h_{8}  = 0.5277 \\
      h_{9}  = 0.5270  & h_{10} = 0.4333 & h_{11} = 0.6774 & h_{12} = 0.5967 & h_{13} = 0.5134 \\
      h_{14} = 0.5549 & h_{15} = 0.4189 & h_{16} = 0.5723 & h_{17} = 0.5904 & h_{18} = 0.5193 \\
      h_{20} = 0.6578 & h_{21} = 0.4295 & h_{22} = 0.4496 & h_{23} = 0.5808 & h_{24} = 0.6469 \\
      h_{26} = 0.6076 & h_{28} = 0.4295 & h_{29} = 0.6087 & h_{30} = 0.3693 & h_{31} = 0.3503 \\
      h_{33} = 0.5403 & h_{34} = 0.4035 & h_{36} = 0.4890 & h_{37} = 0.4295 & h_{39} = 0.6347 \\
      h_{40} = 0.5546 & h_{41} = 0.6544 & h_{42} = 0.6289 & h_{43} = 0.6653.
    \end{matrix}
  \end{displaymath}

\end{theorem}
\begin{proof}
  For each $R_i$ we selected $(a_i,b_i)\in R_i$ as a representative (these choices are shown in Figure~\ref{fig:plat-nums}).  We then computed a combinatorial enclosure via interval arithmetic for $f_{a_i,b_i}$ at different resolutions, and for each resolution we computed a rigorous lower bound for topological entropy using the DFT method (see Section~\ref{sec:dft}).  These bounds are summarized in Figure~\ref{fig:ent-vs-res-2d}.  Finally, by~\cite{zinplats} we know each $R_i$ is (uniformly) hyperbolic and so for each $i$ we can apply the maximum lower bound achieved for $(a_i,b_i)$ to all of $R_i$.
\end{proof}

Figures~\ref{fig:plat-ents-3d} and~\ref{fig:plat-ents-flat} show an overview of our results.  Note that the entropy values shown are merely lower bounds, and not necessarily the true values.  Since we are computing these bounds for hyperbolic parameter values, however, we know from Section~\ref{sec:scaling} that if our entropy lower bound levels off as the grid resolution increases, we have strong evidence that we have obtained the correct value.  Typical index pairs from these computations are shown in Figure~\ref{fig:example-ips}, with colored regions corresponding to symbols in the resulting subshift after  amalgamation using Algorithm~\ref{alg:amalg}.

\begin{figure}[t]
  \centering
  \subfigure[Plateau 12]{
    \includegraphics[width = 0.45\textwidth]{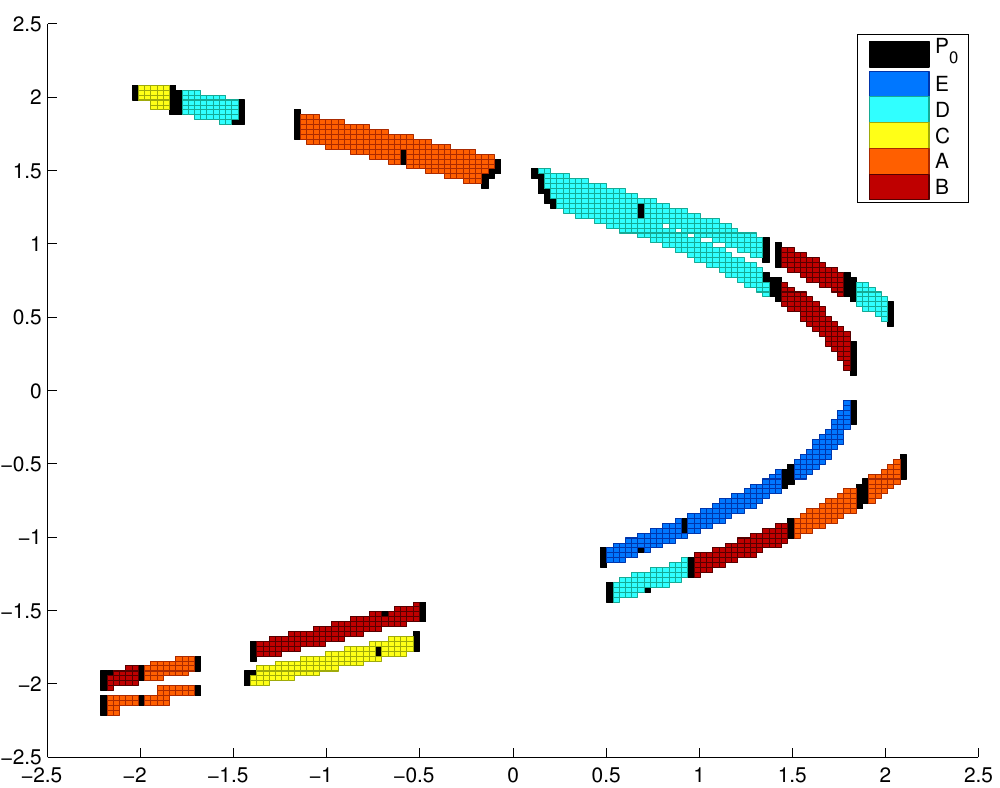}}
  \subfigure[Plateau 24]{
    \includegraphics[width = 0.45\textwidth]{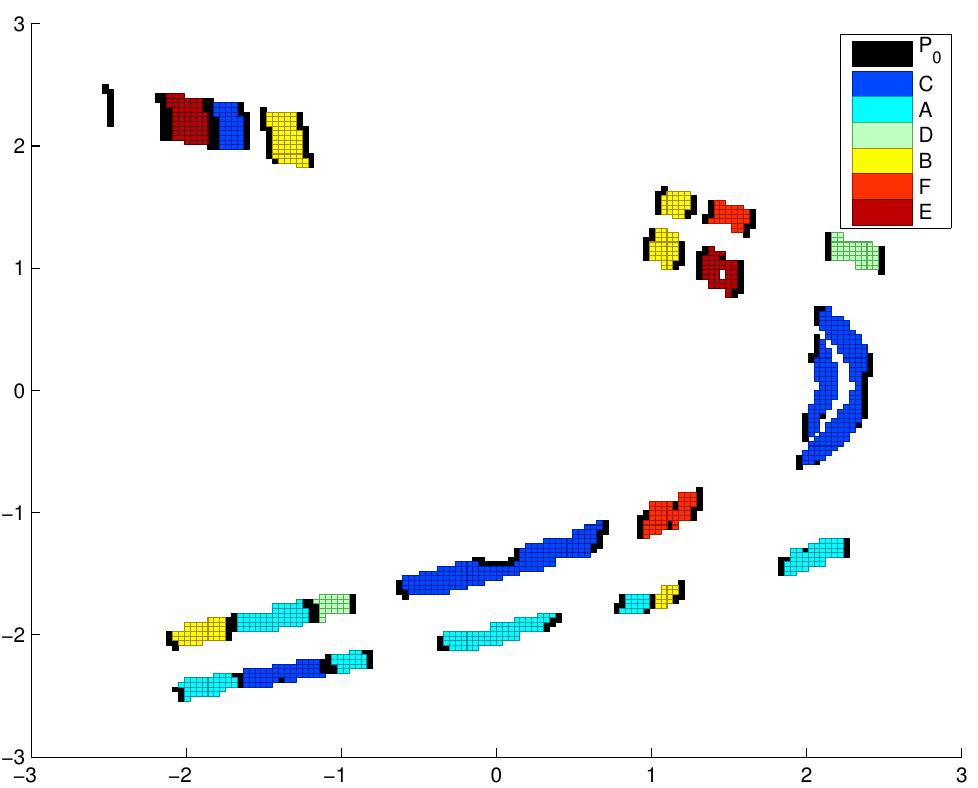}}
  \caption{Two index pairs from the computations.  The black is the set $P_0$ while the colors correspond to the amalgamated regions after applying Algorithm~\ref{alg:amalg}.}
  \label{fig:example-ips}
\end{figure}

In Figure~\ref{fig:ent-vs-res-2d}, we show plots of entropy bounds computed versus the resolution, for each of the 43 parameter values with a nonzero bound.  Using the above heuristic, it seems that for most of the plateaus, the bounds we computed should be exact or very close.  A few notable exceptions are plateaus 9, 10, 21, 28, 30, and 31, since the plots do not seem to have leveled off, and we would expect many of these bounds to improve with further computations.  For plateaus 33, 34, 36, 37, and 39, it is unclear whether the plot has leveled off.  As we saw in Section~\ref{sec:scaling}, the DFT method is fairly robust with respect to the grid resolution, and the entropy lower bound is roughly monotonic in the resolution.  The plots in Figure~\ref{fig:ent-vs-res-2d} reaffirm this, with very few exceptions.

While we have computed a large array of lower bounds, covering a vast portion of the parameter space of the H\'{e}non map, a recent result of Arai gives a method of rigorously computing \emph{exact} entropy values for (uniformly) hyperbolic H\'{e}non~\cite{zinloops}.  In fact, in that work he computes values for plateaus numbered 5, 7, 11, and 12 in Figure~\ref{fig:plat-nums}, which match our lower bounds precisely.  In that it computes exact entropy values, Arai's method is certainly superior to the DFT method in the case we study here, and indeed it would be very interesting to use his method to test the accuracy of our other lower bounds.  However, it is important to note that since Arai's method in~\cite{zinloops} relies heavily on hyperbolicity, it is not as generally applicable as the DFT method; in particular, it could not be directly applied to the H\'{e}non map for the classical parameter values, studied in~\cite{dft}, as the map is not uniformly hyperbolic for those parameters.

While much of the parameter space in Figure~\ref{fig:plat-ents-flat} is covered by our lower bounds, there is still much of the parameter space which is not hyperbolic, or has not yet been proven to be hyperbolic.  Thus, it remains in future work to lower-bound the entropy in the remaining white regions.  At first this seems like a daunting or impossible task, since in the nonhyperbolic regions, we no longer have plateaus of topological entropy, and thus cannot extend a bound from a single set of parameters to an entire region.  Fortunately, the DFT method can be applied to intervals of parameter values $[a_1, a_2]\times[b_1, b_2]$, yielding a single lower bound which applied to the entire interval, as demonstrated in~\cite{dft} and~\cite{std}.  Thus, it would be of great interest to compute lower bounds on an interval tiling of the H\'enon parameter space, and then compare these bounds to those computed here; as mentioned in Section~\ref{sec:implementation}, the automation of the DFT method would enable such parameter explorations.

\subsection{Area-preserving H\'{e}non maps} \label{sec:ap}

When $b=-1$, the H\'{e}non maps are area-preserving and orientation-preserving.  This case has been well studied, especially in the Physics literature.  Starting in 1991, Davis, MacKay, and Sannami (DMS)~\cite{mackay} conjectured that H\'{e}non was hyperbolic for three values of $a$ ($5.4$, $5.59$, $5.65$) and conjectured conjugacies to symbolic dynamics for these three values as well.  In 2002, de Carvalho and Hall in~\cite{howtoprune} replicated the results for $a = 5.4$ using a pruning approach.  In 2004, Hagiwara and Shudo in~\cite{prunehenon} used a different pruning method to replicate all three of the values that DMS studied, and two more.  They also give estimates of the topological entropy for $4 \leq a \leq 5.7$, which are displayed in Figure~\ref{fig:ap-ents}.  Finally in 2007 some rigorous results appeared by Arai in~\cite{zinplats}, where he proved that there are 16 hyperbolic regions for $b=-1$, covering the parameters studied by DMS and the two others studied by Hagiwara and Shudo.  Arai goes on in~\cite{zinloops} to prove that the subshift conjectured by DMS for $a=5.4$ is actually conjugate.

\newcommand{\tdms}{T_{\text{DMS}}} \newcommand{\tdft}{T_{\text{DFT}}}

While our method cannot prove exact topological entropy values or conjugacies, we can attempt to verify that the entropy of the subshifts given in~\cite{mackay} are lower bounds, and perhaps show that the subshifts themselves are semi-conjugate.  We focus first on the $a=5.4$ case, where Davis et al. conjectured that $f_a$ is conjugate to the subshift corresponding to the transition matrix $\tdms$ in Figure~\ref{fig:t-dms}, which has topological entropy $h(\tdms) \approx 0.6774$.  Note that this is the same plateau as plateau 11 from the previous section.

\begin{figure}
  \centering
  \subfigure[The matrix $\tdft$]{%
    \hspace{10pt}
    \includegraphics[width = 0.3\textwidth]{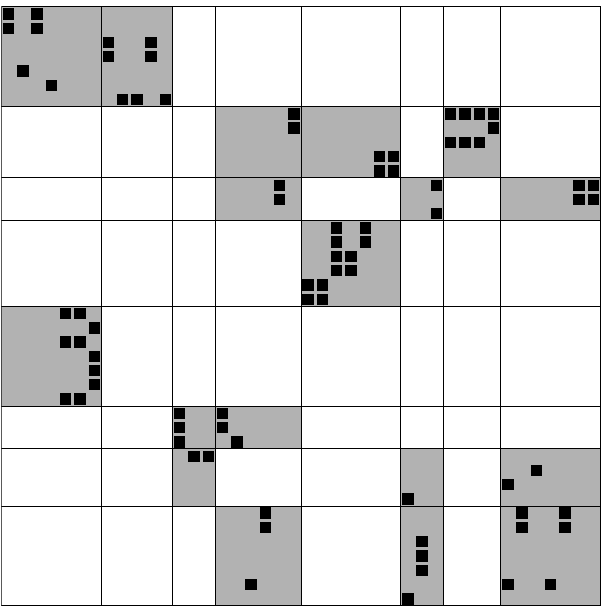}
    \hspace{10pt}
    \label{fig:t-dft}
  }
  \subfigure[The matrix $\tdms$ from~\cite{mackay}]{
    \newcommand{\0}{{ \color{Gray} 0 }}
    \newcommand{\1}{{ \mathbf{1} }}

    \raisebox{65pt}{\large
    \begin{math}
      \begin{bmatrix}
        \1 & \1 & \0 & \0 & \0 & \0 & \0 & \0     \\
        \0 & \0 & \0 & \1 & \1 & \0 & \1 & \0     \\
        \0 & \0 & \0 & \1 & \0 & \1 & \0 & \1     \\
        \0 & \0 & \0 & \0 & \1 & \0 & \0 & \0     \\
        \1 & \0 & \0 & \0 & \0 & \0 & \0 & \0     \\
        \0 & \0 & \1 & \1 & \0 & \0 & \0 & \0     \\
        \0 & \0 & \1 & \0 & \0 & \1 & \0 & \1     \\
        \0 & \0 & \0 & \1 & \0 & \1 & \0 & \1
      \end{bmatrix}
    \end{math}
    }
    \label{fig:t-dms}
  }
  \caption{Amalgamation of the $42\times 42$ symbol matrix obtained using the DFT method.  The black squares in (a) denote the nonzero entries of $\tdft$ while the gray regions represent the amalgamated symbols, which one can easily see match $\tdms$ exactly.}
  \label{fig:dms-amalg}
\end{figure}

Using our technique, we prove that $f_{5.4}$ is semi-conjugate to a $42\times 42$ symbol matrix $\tdft$, depicted in Figure~\ref{fig:t-dft}.  The topological entropy of this matrix is the same value, $h(\tdft) \approx 0.6774$.  The fact that $h(\tdms)=h(\tdft)$ suggests that the subshifts given by $\tdms$ and $\tdft$ might be conjugate, and indeed we prove this in Theorem~\ref{thm:dms}.

\begin{figure}
  \centering
  \includegraphics[width = 0.6\textwidth]{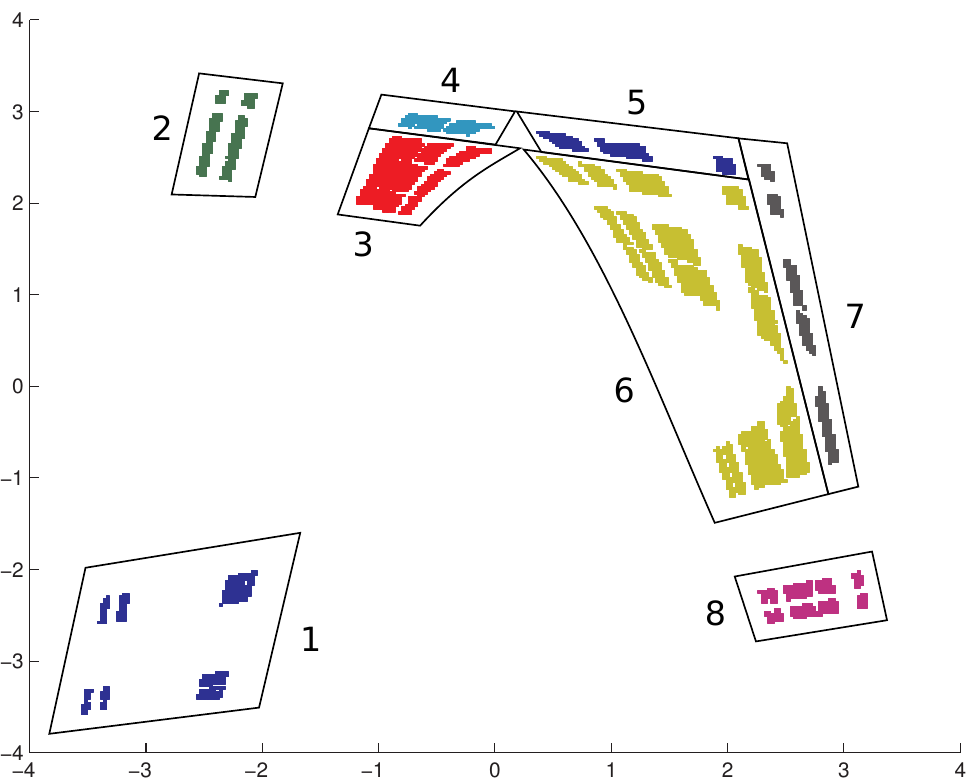}
  \caption{Index pair for $a=5.4$ and $b=-1$, with collections of regions labeled to indicate the symbols for the smaller, shift-equivalent symbol system.}
  \label{fig:mackay-5_4-ip}	  
\end{figure}

\begin{theorem}
  \label{thm:dms}
  The map $f_{5.4}$ is semi-conjugate to the subshift $\tdms$ given in Figure~\ref{fig:t-dms}.  Moreover the symbols of $\tdms$ correspond to the regions labeled in Figure~\ref{fig:mackay-5_4-ip}, which are the same regions conjectured by DMS.
\end{theorem}
\begin{proof}
  Applying Algorithm~\ref{alg:amalg} to $\tdft$, we obtain a strong shift equivalence between $\tdms$ and $\tdft$, which shows that the corresponding subshifts are indeed conjugate.  Moreover, the amalgamated vertices can be chosen so that we obtain the same partition that was used by Davis et al., which is shown in Figure~\ref{fig:mackay-5_4-ip}, with the regions labeled so as to match the symbols (row indices) of $\tdms$.
\end{proof}

\begin{figure}[htb]
  \centering
  \includegraphics[width = 0.8\textwidth]{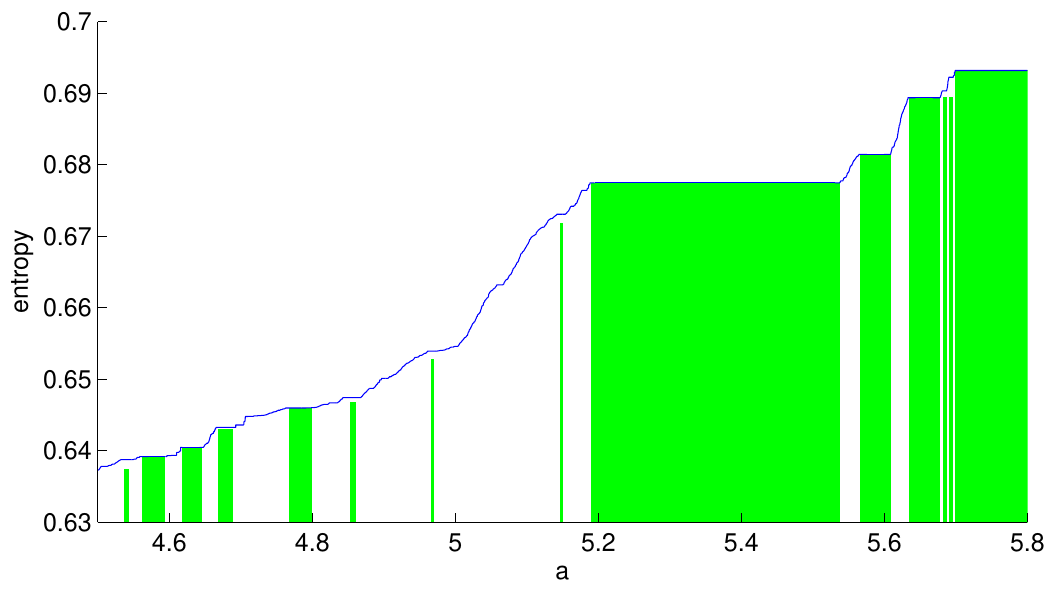}
  \vspace{-10pt}
  \caption{Topological entropy of the area-preserving H\'{e}non maps, where $b=-1$.  The estimates produced by Hagiwara and Shudo are in blue, and our rigorous lower bounds are in green.}
  \label{fig:ap-ents}
\end{figure}

Our method also gives lower bounds which match the values conjectured by Davis et al. for $a=5.59$ and $a=5.65$, as well as the value conjectured by Hagiwara and Shudo for $a=4.58$.  In addition to the these values, we also focus on 11 other values, which all together correspond to the first 15 area-preserving plateaus computed by Arai (the 16th is the maximal entropy plateau, which is plateau 1 in the previous section).  Our results are summarized in the following theorem.

\begin{theorem}
  \label{thm:ap-bounds}
  \vbox{The following entropy bounds hold for the H\'{e}non maps $f_a = f_{a,-1}$.  Here we write $h(f_{[a_0,a_1]}) \geq v$ to mean $\forall a \in [a_0,a_1], h(f_a) \geq v$.}\nobreak
  \begin{displaymath}
    \begin{matrix}
        1.  & h\left(f_{[4.5383, \: 4.5386]}\right) \geq 0.6373 \quad & 2.  & h\left(f_{[4.5388, \: 4.5430]}\right) \geq 0.6373 \\[1pt] 3.  & h\left(f_{[4.5624, \: 4.5931]}\right) \geq 0.6391 \quad & 4.  & h\left(f_{[4.6189, \: 4.6458]}\right) \geq 0.6404 \\[1pt] 5.  & h\left(f_{[4.6694, \: 4.6881]}\right) \geq 0.6429 \quad & 6.  & h\left(f_{[4.7682, \: 4.7993]}\right) \geq 0.6459 \\[1pt] 7.  & h\left(f_{[4.8530, \: 4.8604]}\right) \geq 0.6466 \quad & 8.  & h\left(f_{[4.9666, \: 4.9692]}\right) \geq 0.6527 \\[1pt] 9.  & h\left(f_{[5.1470, \: 5.1497]}\right) \geq 0.6718 \quad & 10. & h\left(f_{[5.1904, \: 5.5366]}\right) \geq 0.6774 \\[1pt] 11. & h\left(f_{[5.5659, \: 5.6078]}\right) \geq 0.6814 \quad & 12. & h\left(f_{[5.6343, \: 5.6769]}\right) \geq 0.6893 \\[1pt] 13. & h\left(f_{[5.6821, \: 5.6858]}\right) \geq 0.6893 \quad & 14. & h\left(f_{[5.6859, \: 5.6860]}\right) \geq 0.6893 \\[1pt] 15. & h\left(f_{[5.6917, \: 5.6952]}\right) \geq 0.6893 \quad
    \end{matrix}   
  \end{displaymath}
\end{theorem}
\begin{proof}
  Using the DFT method (see Section~\ref{sec:dft}), we computed rigorous lower bounds on topological entropy for a single $a$ value for each plateau; the representatives chosen were the following: 4.5385, 4.5409, 4.5800, 4.6323, 4.6788, 4.7838, 4.8600, 4.9679, 5.1483, 5.4000, 5.5900, 5.6500, 5.6839, 5.6859, 5.6934.  Combining these bounds with the hyperbolic plateaus computed in~\cite{zinplats}, and using Theorem~\ref{thm:topentsemiconj}, we can extend each bound to its corresponding plateau.
\end{proof}

Figure~\ref{fig:ap-ents} shows a plot of the lower bounds from Theorem~\ref{thm:ap-bounds}, shown against the estimates computed by Shudo and Hagiwara in~\cite{prunehenon}.  The 4 cases discussed above correspond to plateaus 3, 10, 11, and 12.  For these plateaus, our lower bounds match the estimates exactly, and our bounds for plateaus 4, 5, and 6 are very close.  This is roughly what one would expect given the resolution plots shown in Figure~\ref{fig:ent-vs-res-1d}.  An interesting trend we see in these data is that the algorithm performed better on the larger plateaus.  This is perhaps because the stable and unstable manifolds seem to be more transverse the farther the parameters are from a bifurcation, making it easier to isolate the important regions of the phase space.

\begin{figure}[!ht]
  \centering
  \includegraphics[width = 0.6\textwidth]{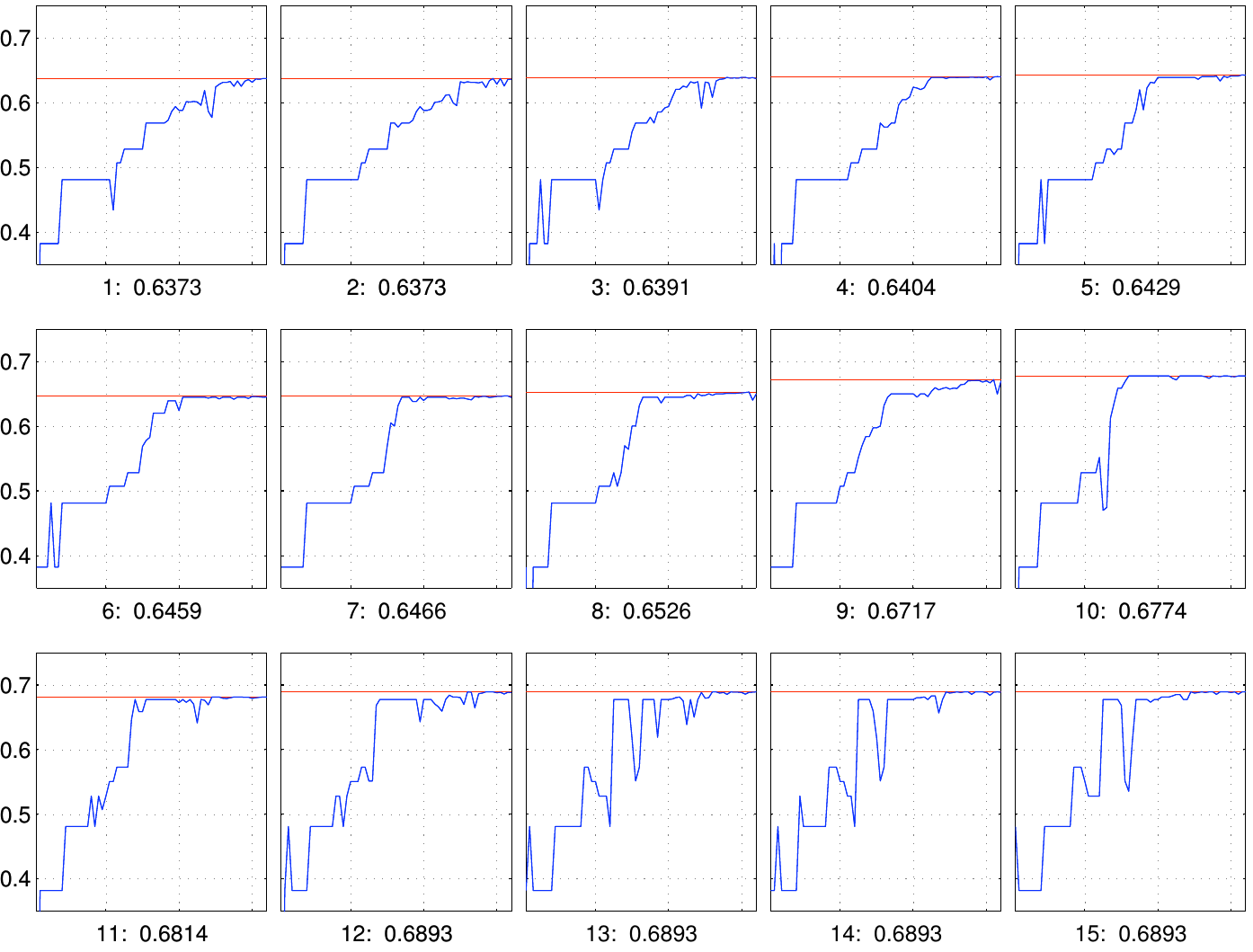}
  \caption{Entropy against resolution, as in
    Figure~\ref{fig:ent-vs-res-2d}.}
  \label{fig:ent-vs-res-1d}
\end{figure}

\subsection*{Acknowledgements} The author would like to thank John Smillie for all of his guidance and advice, Zin Arai for generously sharing his useful data, Sarah Day and Rodrigo Trevi\~no for their helpful comments on numerous drafts of this paper, and Daniel Lepage for his help with the visualization in Figure~\ref{fig:plat-ents-3d}.

\bibliographystyle{plain}
{\small \bibliography{plats-siuro}}
\end{document}